\documentclass[12pt]{amsart}
\usepackage{amsmath,amsthm,amscd,amsfonts, amssymb}
\usepackage{graphicx}
\usepackage{tikz}
\usepackage{color}
\include{diagxy}
\usepgflibrary{shapes}
\newcommand{\sect}[1]{\section{#1}\setcounter{equation}{0}}

\font\mbn=msbm10 scaled \magstep1
\font\mbs=msbm7 scaled \magstep1
\font\mbss=msbm5 scaled \magstep1
\newfam\mbff
\textfont\mbff=\mbn
\scriptfont\mbff=\mbs
\scriptscriptfont\mbff=\mbss\def\mbf{\fam\mbff}
\def\Re{{\mbf R}}
\def\Q{{\mbf Q}}
\def\Z{{\mbf Z}}

\def\N{{\mbf N}}

\newtheorem{Th}{Theorem}[section]

\newtheorem{D}[Th]{Definition}
\newtheorem{Proposition}[Th]{Proposition}
\newtheorem{R}[Th]{Remark}
\newtheorem{E}[Th]{Example}
\newtheorem*{Problem1}{Moments Finiteness Problem}
\newtheorem*{Problem2}{Restricted Moments Finiteness Problem}
\newtheorem*{Theorem}{Theorem}

\begin{document}

\title[]{Moments Finiteness Problem and Center Problem for Ordinary Differential Equations}

\author{Alexander Brudnyi} 
\address{Department of Mathematics and Statistics\newline
\hspace*{1em} University of Calgary\newline
\hspace*{1em} Calgary, Alberta\newline
\hspace*{1em} T2N 1N4}
\email{albru@math.ucalgary.ca}

\thanks{Research supported in part by NSERC}

\begin{abstract}
We study the moments finiteness problem  for the class of Lipschitz maps $F: [a,b]\rightarrow\mathbb R^n$ with images in a compact Lipschitz triangulable curve $\Gamma$. We
apply the obtained results to the center problem for ODEs describing in some cases (including equations with analytic coefficients) the set of universal centers of such equations by vanishing of finitely many moments from their coefficients.
\end{abstract}

\date{} 

\maketitle

\tableofcontents

\sect{Introduction} 
Let $f_1,\dots ,f_n$ be nonconstant real Lipschitz functions on an interval $[a,b]\Subset\Re$.
\begin{D}\label{def1}
A moment of degree $d$ is the expression of the form
\begin{equation}\label{eq1}
\int_a^b f_{1}(t)^{d_1}\cdots f_{n}(t)^{d_n}\cdot f'_{i}(t)\,dt,
\end{equation}
where $1\le i \le n$,  all $d_j\in\Z_+$ and  $\sum_{j=1}^n d_j=d$.
\end{D}
Moments play an important role in the study of the {\em center-focus problem} for a differential equation
\begin{equation}\label{eq2}
\frac{dv}{dt}=\sum_{j=1}^n a_j v^{j+1},\quad\text{where all}\quad a_j\in L^\infty([a,b]),\quad n\in\N\cup\{\infty\}.
\end{equation}
see \cite{AL}, \cite{BlRY}, \cite{BFY}, \cite{BRY}, \cite{BY}, \cite{B1}--\cite{B6}, \cite{C}, \cite{CGM1}--\cite{CGM3}, \cite{CL}, \cite{FPYZ},
\cite{GGL}, \cite{MP}, \cite{P1}, \cite{PRY}, \cite{P2}, \cite{Z} and references therein.  Note that the classical Poincar\'{e} center-focus problem for polynomial vector fields can be equivalently reformulated as the center problem for equations \eqref{eq1} with all $a_j$ being trigonometric polynomials of a special form.

Recall that equation \eqref{eq2} determines a {\em center} if for all sufficiently small initial values $v_0$ the corresponding solution satisfies
$v(b)=v(a):=v_0$. The simplest and, in a certain statistical sense, the most frequently occurring centers, so-called {\em universal centers}, are defined by vanishing of all iterated integrals from the coefficients $a_1,\dots, a_n$.  In this case all possible moments from $f_i(t):=\int_a^t a_i(s)\, ds$, $t\in [a,b]$, $1\le i\le n$, are zeros. 
Conversely, as follows from \cite[Th.~1.2]{B6} if all $a_i$ are real analytic in a neighbourhood of $[a,b]$, then vanishing of all moments from $f_1,\dots, f_n$ implies that the corresponding equation \eqref{eq2} determines a universal center. 

Let $\Gamma_F$ be the image of the Lipschitz map $F:=(f_1,\dots, f_n): [a,b]\rightarrow\Re^n$. By $\mathcal G_F$ we denote the set of Lipschitz maps $G=(g_1,\dots, g_n): [a,b]\rightarrow\Re^n$ such that $G([a,b])\subset\Gamma_F$. Let $\mathcal H\subset \mathcal G_F$ be an infinite subset.

\medskip

In this paper we study the following problem.
\begin{Problem1}
Is there a number $N\in\Z_+$ with the property that if for some $G\in \mathcal H$ all moments of degrees $\le N$ from $g_1,\dots, g_n$ vanish, then all other moments from these functions vanish as well?
\end{Problem1}

Our research is motivated by important results in \cite{CGM3} solving similar moments vanishing problems for maps $F=(f_1,f_2)\rightarrow\Re^2$, where $f_1',f_2'$ are either real univariate polynomials on $[a,b]$ or trigonometric polynomials on $[0,2\pi]$ of degrees $\le d$ (see Remark \ref{rem3.2} below). 
This, in particular, gives an explicit description of the set of universal centers of equation \eqref{eq2} with $n=2$ (the {\em Abel equation}) and $a_1,a_2$ real polynomials or trigonometric polynomials of degrees $\le d$ by means of finitely many algebraic equations from coefficients of $a_1$ and $a_2$.

In the present paper we describe a different from that of \cite{CGM3} approach to the finiteness problem alike the one of \cite{BY}. Specifically,
we will show that under some conditions on $\Gamma_F$ the Moments Finiteness Problem for $\mathcal G_F$ has a positive solution and give an effective bound of the optimal $N$ in question by means of some geometric characteristics of $\Gamma_F$. In Section 4 we apply these results to describe the set of universal centers of certain equations \eqref{eq1} (including equations with analytic coefficients) by vanishing of finitely many moments from their coefficients. \medskip

\noindent {\em Acknowledgment.} I thank Yu. Brudnyi for useful discussions.

\sect{Moments Finiteness Problem for $n=2$} 

Suppose $F:=(f_1,f_2): [a,b]\rightarrow\Re^2$ is a nonconstant Lipschitz map. 
We will assume that the image $\Gamma_F$ of $F$ is a {\em Lipschitz triangulable curve}.
The latter means that $\Gamma_F$ can be presented as the union of finitely many arcs $\gamma_i$ such that for $i\ne j$ the intersection $\gamma_i\cap\gamma_j$ is either empty or consists of at most one of their endpoints, each $\gamma_i$ is parameterized by a Lipschitz map $h_i: [0,1]\rightarrow \gamma_i\subset\Re^2$ such that the inverse map $h_i^{-1}:\gamma_i\rightarrow [0,1]$ is locally Lipschitz apart from the endpoints.
The class of Lipschitz triangulable curves includes, in particular, piecewise $C^1$ curves and images of nonconstant analytic maps $[a,b]\rightarrow\Re^2$, see, e.g. \cite[Ex.\,5.1]{BY}.

According to the definition, $\Gamma_F$ is a finite one-dimensional $CW$-complex and
so it is homotopically equivalent to the wedge sum of circles. 

Our first result solves the Moments Finiteness Problem for $\Gamma_F$ homotopically trivial (for its proof see \cite{B1}).
\begin{Proposition}\label{prop2.1}
If $\Gamma_F$ does not contain subsets homeomorphic to a circle, then for $G=(g_1,g_2)\in \mathcal G_F$ vanishing of moments of degree zero from $g_1, g_2$ imply vanishing of all iterated integrals from these functions.
\end{Proposition}
Now, assume that the fundamental group $\pi_1(\Gamma_F)$ (of loops in $\Gamma_F$ with base point $F(a)\in\Re^2$) is isomorphic, for some $m\in\N$, to the free group with $m$ generators. We will also assume that moments of degree zero from $f_1, f_2$ vanish, or, equivalently, that $F(a)=F(b)$. 
Let $\ell_1,\dots , \ell_m : [0, 1]\rightarrow\Gamma_F$ be simple closed
Lipschitz curves in $\Re^2$ considered with counterclockwise orientation such that their images $[\ell_1],\dots, [\ell_m]$
in $H_1(\Gamma_F)$ are generators of this group. By the Jordan theorem images of curves $\ell_i$ in $\Re^2$ are
boundaries of some simply connected domains $D_i\subset\Re^2$, $1\le i\le m$. Let $A(D_i)$ stand for the area of $D_i$. 

The following result solves 
the Moments Finiteness Problem for a generic $\Gamma_F$ homotopic to the wedge sum of $m$ circles.
\begin{Proposition}\label{prop2.2}
Suppose $A(D_1),\dots, A(D_m)$ are linearly independent over the field $\Q$ of rational numbers. Then for every  Lipschitz map $G=(g_1,g_2): [a,b]\rightarrow\Re^2$, $G(a)=G(b)$, with image in $\Gamma_F$ vanishing of the moment $\int_a^b g_1(t)\cdot g_2'(t)\, dt$ implies vanishing of all other moments from $g_1, g_2$.
\end{Proposition}

Further, consider the family $\{S_j\}_{1\le j\le m}$ of connected components of the open set 
\[
S=\left(\bigcup_{1\le j\le m}D_j\right)\setminus\left(\bigcup_{1\le j\le m}\ell_j([0,1])\right).
\]
Let $r_j$ be the sidelength of the maximal open square with sides parallel to the coordinate axes contained in $S_j$. We set
\[
r_{\Gamma_F}:=\min_{1\le j\le m}\{r_j\} .
\]
By $A_j$ we denote the area of $S_j$ and set
\[
A_{\Gamma_F}:=\max_{1\le j\le m}\{A_j\}.
\]
Finally, by $d_{\Gamma_F}$ we denote half of the sidelength of the minimal closed square with sides parallel to the coordinate axes containing $S$.

Let
\begin{equation}\label{const}
N_{\Gamma_F}:=\left\lfloor\frac{27\pi}{2}\frac{A_{\Gamma_F}d_{\Gamma_F}}{r_{\Gamma_F}^3}\right\rfloor+1.
\end{equation}
One can easily check that $\sqrt{m}\le\frac{d_{\Gamma_F}}{r_{\Gamma_F}}<N_{\Gamma_F}$.

The following result shows that the Moments Finiteness Problem has a positive solution for $F$ as above with the corresponding $N$ bounded by $2N_{\Gamma_F}$.
\begin{Th}\label{te2.1}
There exists a number $\widetilde N\in\N$ such that 
\[
\left\lfloor\sqrt{m}-\frac 12\right\rfloor\le \widetilde N\le N_{\Gamma_F}
\]
and for every  Lipschitz map $G=(g_1,g_2): [a,b]\rightarrow\Re^2$, $G(a)=G(b)$, with image in $\Gamma_F$
each moment from $g_1,g_2$ can be expressed as a finite linear combination with real coefficients  depending on $\Gamma_F$ only of $m$ moments  from the family
\[
\int_a^b g_{1}(t)^{d_1}\cdot g_{2}(t)^{d_2}\cdot g'_{2}(t)\,dt\quad  \text{with}\quad \max\{d_1-1,d_2\}\le \widetilde N.
\]
\end{Th}
It will be proved that a natural number satisfying this property cannot be less than $\lfloor\sqrt{m}-\frac 12\rfloor$. The following example shows that in some cases the above bounds for $\widetilde N$ are almost optimal.
\begin{E}\label{ex2.2}
{\rm
Consider the square $\mathbb K^2:=[-1,1]^2\subset\Re^2$. For a fixed natural number $k$ we partite $\mathbb K^2$ into $k^2$ congruent subsquares with sides parallel to the coordinate axes with sidelength $\frac{2}{k}$. Let $\Gamma_F$ be the union of boundaries of all these subsquares. (Clearly,
$\Gamma_F$ can be obtained as the image of a piecewise linear map $F: [0,1]\rightarrow\Re^2$.) By the definition $\pi_1(\Gamma_F)$ is free group with $k^2$ generators and the family $\{S_j\}_{1\le j\le m}$ consists of all open subsquares in the partition. In particular, $m:=k^2$. Also, for all $j$ we have $r_j=\frac{2}{k}$ and $A_j=r_j^2$, and $d_{\Gamma_F}=1$. Hence, $N_{\Gamma_F}=\lfloor\frac{27\pi k}{4}\rfloor +1$. On the other hand,
$\lfloor\sqrt{m}-\frac 12\rfloor =k-\frac 12$. Thus, $\widetilde N\sim k$.}
\end{E}

\sect{Moments Finiteness Problem for $n\ge 3$} 
Our first result shows that in a sense Moments Finiteness Problems for $n\ge 3$ and $n\ge 2$ are equivalent. 

Suppose that $F:=(f_1,\dots, f_n): [a,b]\rightarrow\Re^n$, $n\ge 3$, is a nonconstant Lipschitz map such that $\Gamma_F:=F([a,b])$ is Lipschitz triangulable. Let $\langle,\cdot,\cdot\rangle$ be the inner product on $\Re^n$. For $G\in \mathcal G_F$ (see the Introduction) by $G_v: [a,b]\rightarrow\Re^2$, $v=(v_1,v_2)\in \Re^n\times\Re^n$, we denote the Lipschitz map $t\mapsto (\langle v_1,G(t)\rangle, \langle v_2,G(t)\rangle)$, $t\in [a,b]$. Next, for a subset $\mathcal H\subset\mathcal G_F$ we define $\mathcal H_v:=\{G_v:=(g_{1v}, g_{2v})\, :\, G\in\mathcal H\}$. Let us consider the following problem for $\mathcal H_v$.
\begin{Problem2}
Is there a number $N_v\in\Z_+$ with the property that if for some $G_v\in\mathcal H_v$ the moments
\[
\int_a^b g_{1v}'(t)\, dt\quad\text{and}\quad \int_a^b g_{1v}(t)^d\cdot g_{2v}'(t)\, dt\quad \text{for all}\quad d\le N_v
\]
vanish, then all other moments of this form from $g_{1v}, g_{2v}$ vanish as well?
\end{Problem2}

\begin{Th}\label{te3.1}
Assume that all maps $G\in\mathcal H\, (\subset\mathcal G_F)$ satisfy $G(a)=G(b)$. Then
for almost all $v\in \Re^n\times\Re^n$ (in the sense of Lebesgue measure) the Moments Finiteness Problem for $\mathcal H$ and the Restricted Moments Finiteness Problem for $\mathcal H_v$ are equivalent and the optimal constants $N$ and $N_v$ in these problems coincide.
\end{Th}
In the proof we describe explicitly the set of points $v\in \Re^n\times\Re^n$ for which Theorem \ref{te3.1} is valid.
\begin{R}\label{rem3.2}
{\rm (1) The theorem shows that if $\Gamma_F\Subset\Re^n$ is such that for almost all $v\in \Re^n\times\Re^n$ the set
$\Gamma_{F_v}\Subset\Re^2$ is Lipschitz triangulable, then one can estimate the optimal constant $N$ in the Moments Finiteness Problem for $\mathcal G_F$ by means of the constant $N_{\Gamma_{F_v}}$ for a generic $v$, see Theorem \ref{te2.1}. This condition is valid if, e.g., $\Gamma_F$ is a piecewise $C^1$ curve (a consequence of the Sard theorem) or the image of a nonconstant analytic map $[a,b]\rightarrow\Re^n$. 

(2) Theorem \ref{te3.1} allows to generalize the main results of \cite{CGM3} to the multidimensional case. Specifically, we have \medskip

\noindent {\bf Theorem.} {\em If $f_1',\dots, f_n'$ are either real univariate polynomials on $[a,b]$ or trigonometric polynomials on $[0,2\pi]$ of degrees $\le d$, then
vanishing of all moments from $f_1,\dots, f_n$ of degrees $2d-3$ in the first case and of degrees $4d-3$ in the second one implies vanishing of all other moments from these functions.}
\medskip

This result follows from \cite{CGM3} and the fact that for each $v\in\Re^n\times\Re^n$ the corresponding functions $f_{1v}', f_{2v}'$ are either polynomials or trigonometric polynomials of degrees $\le d$.
}
\end{R}

Let us present some other results providing positive solutions of the Moments Finiteness Problem. (As before we assume that
$\Gamma_F$ is Lipschitz triangulable.)

The following result is the extension of Proposition \ref{prop2.1}. It solves the problem for $\Gamma_F$ homotopically trivial (we refer \cite{B1} for the proof).
\begin{Proposition}\label{prop3.1}
If $\Gamma_F$ does not contain subsets homeomorphic to a circle, then for $G=(g_1,\dots, g_n)\in \mathcal G_F$ vanishing of moments of degree zero from $g_1,\dots, g_n$ imply vanishing of all iterated integrals from these functions.
\end{Proposition}

Thus, as in the previous section, we may assume that the fundamental group $\pi_1(\Gamma_F)$ (of loops in $\Gamma_F$ with base point $F(a)\in\Re^n$) is isomorphic, for some $m\in\N$, to the free group with $m$ generators. Let $C_1,\dots, C_m\subset\Gamma_F$ be images of simple closed Lipschitz curves which generate group $H_1(\Gamma_F)$. Since $\Gamma_F$ is Lipschitz triangulable, there exist
Lipschitz curves $h_1,\dots , h_k : [0, 1]\rightarrow\cup_{1\le j\le m}C_j=:C$ 
forming the triangulation $\mathcal T$ of $C$. We set $\gamma_i:=h_i([0,1])$, $1\le i\le k$. \medskip

\noindent {\em Notation.} In what follows all cubes in $\Re^n$ are assumed to have sides parallel to the coordinate axes. The radius of such a cube is half of its sidelength.

\medskip

Assume that $\Gamma_F$ satisfies the following property: 
\begin{itemize}
\item[(*)]
There exist families of pairwise disjoint open cubes $K_i\subset\Re^n$ with centers $c_i\in\gamma_i$ and natural numbers $s_i\in [1,n]$,
$1\le i\le k$, such that for each $i$

\noindent (a) $K_i\cap\gamma_i$ is connected  and $K_i\cap\gamma_j=\emptyset$ for all $j\ne i$;

\noindent (b) restriction to $K_i\cap\gamma_i$ of the projection $\pi_{s_i}:\Re^n\rightarrow\Re^n$,
$\pi_{s_i}\bigl((x_1,\dots, x_n)\bigr):=(0,\dots, 0,x_{s_i},0,\dots, 0)$, is injective.
\end{itemize}

(This condition is valid, e.g., if $\Gamma_F$ is piecewise $C^1$, or the image of a nonconstant analytic map $[a,b]\rightarrow\Re^n$.)

Let $r_i$ be the radius of $K_i$ and $l_i$ be linear measure of the set $\pi_{s_i}\bigl(\frac 12 K_i\cap\gamma_i\bigr)$, where $\frac 12 K_i:=\frac 12 (K_i-c_i)+c_i$ (the cube $\frac 12$-homothetic to $K_i$ with respect to $c_i$). We set
\[
r_{\mathcal T}:=\min_{1\le i\le k}\{r_i\}\quad\text{and}\quad l_{\mathcal T}:=\min_{1\le i\le k}\{l_{s_i}\}.
\]
(According to condition (*), $l_{\mathcal T}>0$.)

Also, we denote 
\[
L_i:=\int_0^1 \|h_i'(t)\|_{\ell_1^n}\, dt,\qquad  L_{\mathcal T}=\max_{1\le i\le k} L_i;
\]
(Here $\|(x_1,\dots, x_n)\|_{\ell_1^n}:=\sum_{j=1}^n |x_j|$.)

Finally, by $d_{\Gamma_F}$ we denote the radius of the minimal closed cube containing $C$.

Let
\begin{equation}\label{const1}
\bar{N}_{\mathcal T}:=\left\lfloor\frac{2\pi n L_{\mathcal T} d_{\Gamma_F}}{r_{\mathcal T} l_{\mathcal T}}\right\rfloor+1.
\end{equation}

The following result shows that the Moments Finiteness Problem has a positive solution for $F$ as above with the corresponding $N$ bounded by $n \bar{N}_{\mathcal T}$.
\begin{Th}\label{te3.4}
There exists a number $\widetilde N\in\N$ such that 
\[
\max\left\{\left\lfloor\sqrt[n]{\frac mn}-1\right\rfloor, 1\right\}\le \widetilde N\le N_{\mathcal T}
\]
and for every  Lipschitz map $G=(g_1,\dots,g_n): [a,b]\rightarrow\Re^n$, $G(a)=G(b)$, with image in $\Gamma_F$
each moment from $g_1,\dots,g_n$ can be expressed as a finite linear combination with real coefficients depending on $\Gamma_F$ only of $m$ moments from the family
\[
\int_a^b g_1(t)^{d_1}\cdots g_n(t)^{d_n}\cdot g_i'(t)\, dt\quad  \text{with}\quad \max_{1\le j\le n}\{d_j\}\le \bar{N}_{\mathcal T},\quad 1\le i\le n.
\]
\end{Th}
The next example generalizes Example \ref{ex2.2} and shows that in some cases the bounds for $\widetilde N$ given in the previous theorem are almost optimal (up to constants depending on $n$).
\begin{E}\label{ex3.5}
{\rm Consider the cube $\mathbb K^n:=[-1,1]^n\subset\Re^n$. For a fixed natural $k$ we partite $\mathbb K^n$ into $k^n$ congruent subcubes of radius $\frac 1k$. Let $\Gamma_F$ be the union of edges of all these subcubes. (Clearly, $\Gamma_F$ can be obtained as the image of a piecewise linear map $F: [0,1]\rightarrow\Re^n$.) The rank $m$ of $H_1(\Gamma_F)$ is the number of square faces of the $CW$ complex consisting of the union of all subcubes in the partition. One has the following estimates for $m$:
\[
n\cdot (n-1)\cdot 2^{n-3}\cdot k^n\cdot\frac{1}{2^{n-2}}<m<n\cdot (n-1)\cdot 2^{n-3}\cdot (k+2)^n\cdot\frac{1}{2^{n-2}}.
\]
Hence,
\begin{equation}\label{eq3.2}
\left\lfloor\sqrt[n]{\frac{n-1}{2}}\cdot k -1\right\rfloor < \left\lfloor\sqrt[n]{\frac mn}-1\right\rfloor<\left\lfloor\sqrt[n]{\frac{n-1}{2}}\cdot (k+2)-1\right\rfloor.
\end{equation}
Next, we consider the triangulation $\mathcal T$ of $\Gamma_F$ consisting of parameterized edges of the above partition. We easily obtain that
$r_{\mathcal T}=\frac{1}{2k}$, $l_{\mathcal T}=\frac{1}{4k}$, $L_{\mathcal T}=\frac 2k $ and $d_{\Gamma_F}=1$. Thus,
\begin{equation}\label{eq3.3}
\bar{N}_{\mathcal T}=\left\lfloor\frac{2\pi n\cdot\frac{2}{k}}{\frac{1}{8k^2}}\right\rfloor +1=\lfloor 32\pi n k\rfloor +1.
\end{equation}
From \eqref{eq3.2} and \eqref{eq3.3} we deduce that $\widetilde N= c(n)k$, where $0<c_1\le c(n)\le c_2 n$ for some absolute constants $c_1, c_2$.
}
\end{E}

\sect{Applications to Center Problem for ODEs}
\subsection{Traversable and Unicursal Curves}
An {\em Eulerian trail} in an undirected connected graph is a path that uses each edge exactly once. If such a path exists, the graph is called {\em traversable}. If, in addition, it is closed, the graph is called {\em unicursal}.

It is well known that a connected undirected graph is traversable if and only if at most two of its vertices have odd degree and is unicursal if 
every its vertex has even degree.

A connected one-dimensional $CW$-complex $X$ is said to be traversable or unicursal if it is homeomorphic to a traversable or unicursal graph. In this case the notion of an Eulerian trail $\gamma : I\rightarrow X\, (:=\gamma(I))$, where $I$ is a closed arc of the unit circle $\mathbb S$, is well defined. ($I:=\mathbb S$ iff $X$ is unicursal). We say that a path $\tilde\gamma: [a,b]\rightarrow X$ covers $\gamma$ if there exists a path $\gamma': [a,b]\rightarrow I$ such that $\tilde\gamma=\gamma\circ\gamma'$. 

The next result is a straightforward corollary of \cite[Th.\,1.5]{B6}.
\begin{Th}\label{ex4.1}
Let $F: [a,b]\rightarrow\Re^n$ be a nonconstant analytic map defined in a neighbourhood of $[a,b]$. Then the Lipschitz triangulable curve $\Gamma_F:=F([a,b])$ is traversable. Moreover, $F:[a,b]\rightarrow\Gamma_F$ covers an Eulerian trail in $\Gamma_F$.
\end{Th}

In our applications we use the following result.

\begin{Th}\label{teo4.2}
Let $F=(f_1,\dots, f_n): [a,b]\rightarrow\Re^n$ be a Lipschitz map such that $\Gamma_F$ is Lipschitz triangulable and traversable and $F$ covers an Eulerian trail in $\Gamma_F$. The following conditions are equivalent:
\begin{itemize}
\item[(a)]
The path $F:[a,b]\rightarrow\Gamma_F$ is closed (i.e., $F(a)=F(b)$) and represents the unit element of $\pi_1(\Gamma_F)$ (i.e., it is contractible in $\Gamma_F$);
\item[(b)]
$F$ satisfies the composition condition, i.e., $F=\tilde F\circ\tilde f$ for continuous maps $\tilde F: [a,b]\rightarrow\Gamma_F$ and $\tilde f: [a,b]\rightarrow [a,b]$ such that $\tilde f(a)=\tilde f(b)$;
\item[(c)]
For all possible $i_1,\dots, i_k\in\{1,\dots, n\}$ the iterated integrals
\[
\int\cdots\int_{a\leq t_{1}\leq\cdots\leq t_{k}\leq b}f'_{i_{k}}(t_{k})\cdots f'_{i_{1}}(t_{1})\ \!dt_{k}\cdots dt_{1}
\]
are equal to zero;
\item[(d)]
The path $F:[a,b]\rightarrow\Gamma_F$ is closed and represents an element of the commutator subgroup $[\pi_1(\Gamma_F),\pi_1(\Gamma_F)]\subset\pi_1(\Gamma_F)$;
\item[(e)]
For all possible $d_1,\dots , d_n\in\Z_+$ and $i\in\{1,\dots, n\}$ the moments 
\[ 
\int_a^bf_1(t)^{d_1}\cdots f_n(t)^{d_n}f'_{i}(t)\,dt
\]
are equal to zero.
\end{itemize}
\end{Th}
Equivalence of (a) and (c) for Lipschitz maps $F$ with Lipschitz triangulable images $\Gamma_F$ is established in \cite[Cor.\,1.12,\, Th.\,1.14]{B1}. In this setting, these conditions are equivalent to a weaker condition, the so-called {\em tree composition condition}, see \cite[Cor.\,3.5]{BY},
\begin{itemize}
\item[(b$'$)]
{\em There exist a finite tree $T$ and continuous maps $\tilde F: T\rightarrow\Gamma_F$ and $\tilde f: [a,b]\rightarrow T$ with $\tilde f(a)=\tilde f(b)$ such that $F=\tilde F\circ\tilde f$.}
\end{itemize}
Thus, (b) implies (a) and (c).

Clearly, (a) implies (d). Equivalence of (d) and (e)  for Lipschitz maps with Lipschitz triangulable images is proved in \cite[Cor.\,3.11]{B2}. Finally, in \cite[Sec.\,3.1]{B6} the implication (d)$\Rightarrow$(b) is proved for $F$ analytic. The proof in the general case goes along the same lines and will be describes in Section 9. This will complete the proof of Theorem \ref{teo4.2}.

Note that if $\pi_1(\Gamma_F)$ is free group with at least two generators, then one can easily construct a closed Lipschitz path $[a,b]\rightarrow\Gamma_F$ which represents an element of $[\pi_1(\Gamma_F),\pi_1(\Gamma_F)]$ but is not contractible. Thus, (d)$\not\Rightarrow$(a) for generic $F$.
\subsection{Center Problem for ODEs} 
In this section we formulate several applications of our main results to the center problem for ODEs.

For real nonconstant Lipschitz functions $f_1,\dots, f_n$ on $[a,b]$ equal to zero at $a$, consider the differential equation
\begin{equation}\label{eq4.3}
\frac{dv}{dt}=\sum_{j=1}^n f_j' v^{j+1}.
\end{equation}

We will assume that the image $\Gamma_F$ of the Lipschitz map $F=(f_1,\dots, f_n): [a,b]\rightarrow\Re^n$ is Lipschitz triangulable.

We say that equation \eqref{eq4.3} determines a {\em universal center} if $F$ satisfies one of the equivalent conditions (a), (b$'$) or (c).

The following result is an immediate consequence of Proposition \ref{prop3.1}.

\begin{Th}\label{teo4.3.1}
If $\Gamma_F$ does not contain subsets homeomorphic to a circle and $F$ is a closed path in $\Gamma_F$, then the corresponding to $F$ equation \eqref{eq4.3} determines a universal center.
\end{Th}

Suppose $\Gamma\Subset\Re^n$ is a compact Lipschitz triangulable curve. By $P(\Gamma)$ we denote the set of closed Lipschitz paths $[a,b]\rightarrow\Gamma$ that cover Eulerian trails of traversable subcurves of $\Gamma$. We set 
\[
P_n:=\bigcup_{\Gamma\Subset\Re^n}P_n(\Gamma).
\]

Our next result relates Proposition \ref{prop2.2} to Center Problem for Abel differential equations. We retain notation of this proposition.

Let $\mathcal A\subset P_2$ consist of paths in curves $\Gamma\Subset\Re^2$ for which areas of (simply connected) domains $D_i$ with boundaries $[\ell_i]$ generating $H_1(\Gamma)$ are linearly independent over $\Q$.
\begin{Th}\label{teo4.4.1}
Consider the Abel equation on $[a,b]$
\begin{equation}\label{eq4.2}
\frac{dv}{dt}=f_1' v^{2}+f_2' v^3,
\end{equation}
where $F:=(f_1,f_2):[a,b]\rightarrow\Gamma_F\Subset\Re^2$, $F(a)=F(b)=0$, belongs to $\mathcal A$.
Then this equation determines a center if and only if 
\[
\int_a^b f_1(t)\cdot f_2'(t)\, dt=0.
\]
Moreover, all centers of \eqref{eq4.2} are universal.
\end{Th}

For a compact Lipschitz triangulable curve $\Gamma\Subset\Re^2$ with $\pi_1(\Gamma)\ne\{1\}$ by $N_\Gamma\in\N$ we denote the number defined in \eqref{const} (with $\Gamma_F$ replaced by $\Gamma$). We also set $N_\Gamma:=0$ for curves $\Gamma$ with $\pi_1(\Gamma)=\{1\}$. Let $C_2(d)$ consist of all Lipschitz triangulable curves $\Gamma\Subset\Re^2$ with $N_\Gamma\le d$. By $C_2$ we denote the set of all compact Lipschitz triangulable curves in $\Re^2$. Then we have $C_2(d_1)\subseteq C_2(d_2)$ for $d_1\le d_2$ and $C_2=\cup_{d\in\Z_+}C_2(d)$.
We set
\[
P_2(d):=\bigcup_{\Gamma\in C_2(d)}P_2(\Gamma).
\]

\begin{Th}\label{teo4.5}
The set of universal centers of equations \eqref{eq4.2} with $F=(f_1,f_2)\in P_2(d)$ is defined by vanishing of moments
\[
\int_a^b f_{1}(t)^{d_1}\cdot f_{2}(t)^{d_2}\cdot f'_{2}(t)\,dt\quad  \text{with}\quad \max\{d_1-1,d_2\}\le d.
\] 
\end{Th}

To formulate a similar result in $\Re^n$, we restrict ourselves to the class $C_n^*$ of compact Lipschitz triangulable curves $\Gamma\Subset\Re^n$ satisfying condition (*). For each $\Gamma\in C_n^*$ with $\pi_1(\Gamma)\ne\{1\}$ the number $\bar{N}_{\mathcal T}$ depending on the fixed triangulation $\mathcal T$ of $\Gamma$ is given by equation \eqref{const1} (with $\Gamma_F$ replaced by $\Gamma$). We define
\[
\bar{N}_{\Gamma}:=\min_{\mathcal T}\bar{N}_{\mathcal T},
\]
where the minimum is taken over all triangulations satisfying condition (*).

We also set $\bar{N}_\Gamma:=0$ for curves $\Gamma$ with $\pi_1(\Gamma)=\{1\}$.  Let $C_n^*(d)$ consist of all curves $\Gamma\in C_n^*$ with $\bar{N}_\Gamma\le d$. Then we have $C_n^*(d_1)\subseteq C_n^*(d_2)$ for $d_1\le d_2$ and $C_n^*=\cup_{d\in\Z_+}C_n^*(d)$.
We set
\[
P_n^*(d):=\bigcup_{\Gamma\in C_n^*(d)}P_n(\Gamma).
\]
\begin{Th}\label{teo4.6}
Consider equations \eqref{eq4.3} such that $F=(f_1,\dots, f_n): [a,b]\rightarrow\Re^n$, $F(a)=F(b)=0$, belong to $P_n^*(d)$. Then 
the set of universal centers of this class of equations is defined by vanishing of moments
\[
\int_a^b f_1(t)^{d_1}\cdots f_n(t)^{d_n}\cdot f_i'(t)\, dt\quad  \text{with}\quad \max_{1\le j\le n}\{d_j\}\le d,\quad 1\le i\le n.
\]
\end{Th}
\begin{R}
{\rm (1) If in Theorems \ref{teo4.5}, \ref{teo4.6} $F$ belongs to the class of rectangular paths, that is, the image $\Gamma_F$ of $F$ belongs to the union of compact intervals translated from coordinate axes, then conditions of the theorems are necessary and sufficient for such an $F$ to determine a center of the corresponding differential equation. This follows from the result proved in \cite{B4} asserting that any center of equation \eqref{eq4.3} determining a rectangular path is universal. 

(2) If we consider equation \eqref{eq4.3} such that $f_1',\dots, f_n'$ are either real univariate polynomials on $[a,b]$ or trigonometric polynomials on $[0,2\pi]$ of degrees $\le d$, then this equation determines a universal center if and only if all moments from $f_1,\dots, f_n$ of degrees $2d-3$ in the first case and of degrees $4d-3$ in the second one vanish. This is the consequence of the theorem formulated in Remark \ref{rem3.2}  and Theorem \ref{teo4.2}.
}
\end{R}

\sect{Polynomial Approximation of Lipschitz Functions}
In this section we present an auxiliary result used in the subsequent proofs.

Let $f: \mathbb K^n\rightarrow\Re$ be a Lipschitz function whose Lipschitz constant is bounded by $L$ over each line parallel to a coordinate axis,
i.e., $|f(x)-f(y)|\le L\|x-y\|_{\ell_\infty^n}$ for all $x,y\in L$.
\begin{Proposition}\label{approx}
There exists a polynomial $p_k$ of degree $k$ in each variable such that
\[
\|f-p_k\|_{C(\mathbb K^n)}\le\frac{\pi n}{k}L.
\]
\end{Proposition}
\begin{proof}
Set
\begin{equation}\label{hat}
\hat f(x):=f(\cos x_1,\dots,\cos x_n),\quad x=(x_1,\dots, x_n)\in [-\pi,\pi]^n.
\end{equation}
Then $f$ is even $2\pi$-periodic in each variable.

We use the following result, see \cite[pp.375--376]{A}.
\medskip

\begin{Theorem} 
Let $h:\Re\rightarrow\Re$ be a continuous $2\pi$-periodic function. There exists a bounded linear operator $T_k\in \mathcal B(C[-\pi,\pi],\mathcal T_k)$ of norm one, where $\mathcal T_k\subset C[-\pi,\pi]$ is the subspace of trigonometric polynomials of degree $k$, such that
\[
\|h-T_k h\|_{C[-\pi,\pi]}\le 2\omega\left(\frac{\pi}{2k};h\right).
\]
If $h$ is even, then $T_k h$ is.
\end{Theorem}
(Here $\omega(\cdot\,; h)$ is the modulus of continuity of $h$.)

\medskip

Let $T_k^i: C([-\pi,\pi]^n)\rightarrow C([-\pi,\pi]^n)$ be a bounded linear operator of norm one defined by the application of the operator $T_k$ to the $i\,$th coordinate of a function, i.e., 
\[
(T_k^i g)(x_1,\dots, x_n):=\left(T_k g(x_1,\dots, x_{i-1},\cdot,x_{i+1},\dots, x_n)\right)(x_i),\quad g\in C([-\pi,\pi]^n).
\]
Then 
\[
T_k^i T_k^j=T_k^j T_k^i\quad\text{for all}\quad 1\le i,j\le n,
\]
and $T_k^{(n)}:=T_k^1\cdots T_k^n$ is a bounded linear operator of norm one on $C([-\pi,\pi]^n)$ with range in the subspace of trigonometric polynomials on $\Re^n$ of degree $k$ in each variable. Moreover, if $g$ is even in each variable, then $T_k^{(n)}g$ is even in each variable as well. In particular, this is valid for $T_k^{(n)}\hat f$, where $\hat f$ is defined in \eqref{hat}. Moreover, since
$|\cos t'-\cos t''|\le |t'-t''|$ for all $t',t''\in\Re$, the above Theorem implies that
\[
\|\hat f-T_k^i\hat f\|_{C([-\pi,\pi]^n)}\le \sup_{\hat x_i\in [-\pi,\pi]^{n-1}}\left\{2\omega\left(\frac{\pi}{2k};\hat f(x_1,\dots, x_{i-1},\cdot,x_{i+1},\dots, x_n)\right)\right\}\le\frac{\pi}{k}L,
\]
where $\hat x_i:=(x_1,\dots, x_{i-1},x_{i+1},\dots, x_n)$.

Finally, since $\|T_k^i\|=1$ for all $1\le i\le n$,
\[
\begin{array}{l}
\displaystyle
\|\hat f-T_k^{(n)}\hat f\|_{C([-\pi,\pi]^n)}\le \|\hat f-T_k^1\hat f\|_{C([-\pi,\pi]^n)}+\|T_k^1\|\cdot\|\hat f-T_k^2\hat f\|_{C([-\pi,\pi]^n)}+\cdots\\ \\
\displaystyle
+\left\|\prod_{i=1}^{n-1}T_k^i\right\|_{C([-\pi,\pi]^n)}\cdot \|\hat f- T_k^n\hat f\|_{C([-\pi,\pi]^n)}\le \frac{\pi n}{k}L.
\end{array}
\]
Passing from $\hat f$ to $f$ we get 
\[
\sup_{(t_1,\dots, t_n)\in \mathbb K^n}\left\{|f(t_1,\dots, t_n)-(T_k^{(n)}\hat f)(\cos^{-1} t_1,\dots,\cos^{-1} t_n)|\right\}\le \frac{\pi n}{k}L.
\]

But $T_k^{(n)}\hat f$ is even in each variable and therefore it is a polynomial in $\cos (s\,x_i)$, $1\leq s\le k$, $1\le i\le n$. Since $\cos(s\cos^{-1}x)=T_s(x)$ (the Chebyshev polynomial of degree $s$), $p_k(t_1,\dots, t_n):=(T_k^{(n)}\hat f)(\cos^{-1} t_1,\dots,\cos^{-1} t_n)$ is a polynomial of degree $k$ in each variable.
\end{proof}
\sect{Proofs of Proposition \ref{prop2.2} and Theorem \ref{te2.1}}
\subsection{Formulae for Moments}
Let $G:[a,b]\rightarrow\Gamma_F\Subset\Re^2$ be a closed Lipschitz path.
We have
\begin{equation}\label{eq2.1}
\int_a^b g_1(t)^{d_1}\cdot g_2(t)^{d_2}\cdot g_i'(t)\, dt=\int_G x_1^{d_1}x_2^{d_2} dx_i,\quad i=1,2.
\end{equation}
Here $x_1,x_2$ are coordinates on $\Re^2$ and $\int_G$ is the integral over the path $G$ (it is well defined because $G$ is Lipschitz).
Let $[G]$ be the image of the path $G$ in $H_1(\Gamma_F)\, (\cong \pi_1(\Gamma_F)/[\pi_1(\Gamma_F),\pi_1(\Gamma_F)])$.
Then $[G]=\sum_{j=1}^m n_j [\ell _j]$ for some $n_j\in\Z$, and since the integrands on the right-hand sides of \eqref{eq2.1} are $d$-closed forms on $\Gamma_F$ and curves $\ell_j$ are Lipschitz, 
\begin{equation}\label{eq2.2}
\int_G x_1^{d_1}x_2^{d_2} dx_i=\sum_{j=1}^m n_j\cdot\int_{\ell_j}x_1^{d_1}x_2^{d_2} dx_i,\quad i=1,2.
\end{equation}
Applying the Green formula to the integrals on the right-hand sides of the previous equation (it is possible by the definition of Lipschitz triangulable curves) we finally obtain for $k:=i+(-1)^{i+1}$
\begin{equation}\label{equ2.3}
\int_a^b g_1(t)^{d_1}\cdot g_2(t)^{d_2}\cdot g_i'(t)\, dt=\sum_{j=1}^m n_j\cdot\int\int_{D_j}(-1)^i d_i\cdot x_i^{d_i-1}x_k^{d_k}\,dx_1 dx_2.
\end{equation}
\subsection{Proof of Proposition \ref{prop2.2}}
The hypothesis of the proposition and \eqref{equ2.3} with $d_1=1$, $d_2=0$ imply that
\[
0=\int_a^b g_1(t)\cdot g_2'(t)\, dt=\sum_{j=1}^m n_j\cdot\int\int_{D_j}dx_1 dx_2=\sum_{j=1}^m n_j\cdot A(D_j).
\]
Since $A(D_1),\dots, A(D_m)$ are linearly independent over the field of rational numbers, all $n_j=0$. Then we deduce from equations \eqref{equ2.3} that all moments from $g_1,g_2$ vanish. 

\subsection{Proof of Theorem \ref{te2.1}}
We can reorganize the sum of the integrals on the right-hand side of \eqref{equ2.3} to get
\begin{equation}\label{eq2.3}
\sum_{j=1}^m n_j\cdot\int\int_{D_j}(-1)^i d_i\cdot x_i^{d_i-1}x_k^{d_k}\,dx_1 dx_2=\sum_{j=1}^m \tilde n_j\cdot\int\int_{S_j}(-1)^i d_i\cdot x_i^{d_i-1}x_k^{d_k}\,dx_1 dx_2,
\end{equation}
where $\tilde n_j\in\Z$ are sums of some elements of the family $\{n_i\}_{1\le i\le m}$. In particular, all $\tilde n_j=0$ if and only if all $n_j=0$.

Let $d_{S_j}$ be the distance to the closed set $\Re^2\setminus S_j$ in the $\ell_\infty^2$ metric on $\Re^2$.
By the definition $d_{S_j}$ is a Lipschitz function with respect to this metric with Lipschitz constant one. We set 
\[
h_j:={\rm sgn}(\tilde n_j)\cdot d_{S_j}\quad\text{ and}\quad h:=\sum_{j=1}^m h_j.
\]

By the definition of the function $d_{S_j}$ and the Rademacher theorem the derivative of $h_j$ restricted to a line $L$ parallel to one of the coordinate axes exists a.e. and its modulus is bounded by one on $L\cap S_j$ and equals zero on $L\setminus S_j$. Thus the derivative of $h|_L$ exists a.e. and its modulus is bounded by one on $L\cap S$, $S:=\sqcup_{j=1}^m S_j$, and equals zero on $L\setminus S$. This implies that $h|_L$ is Lipschitz with Lipschitz constant bounded by one. Applying to $h$ Proposition \ref{approx} we find a sequence of real polynomials $\{p_k\}_{ k\in\N}$, where $p_k$ is of degree $k$ in each variable, such that
\[
\|h-p_k\|_{C(Q)}\le \frac{2\pi d_{\Gamma_F}}{k}.
\]
(Here $Q$ is the minimal closed square, cf. Notation in Section 3, containing all $S_j$.)

Hence, we have
\[
\begin{array}{l}
\displaystyle
\sum_{j=1}^m |\tilde n_j|\cdot \int\int_{S_j} d_{S_j}(x_1,x_2)\, dx_1dx_2=
\left|\sum_{j=1}^m \tilde n_j\cdot \int\int_{S_j} h(x_1,x_2)\, dx_1dx_2\right|\\
\\
\displaystyle
\le \sum_{j=1}^m |\tilde n_j|\cdot \int\int_{S_j} |h(x_1,x_2)-p_{k}(x_1,x_2)|\, dx_1dx_2
+\left|\sum_{j=1}^m \tilde n_j \int\int_{S_j} p_k(x_1,x_2)\, dx_1dx_2\right|\\
\\
\displaystyle
\le \frac{2\pi d_{\Gamma_F}}{k}\left(\sum_{j=1}^m |\tilde n_j|\cdot A_j\right)+
\left|\sum_{j=1}^m \tilde n_j \int\int_{S_j} p_k(x_1,x_2)\, dx_1dx_2\right|.
\end{array}
\]

Let $Q_j$ be the  maximal open square contained in $S_j$. By $cQ_j$, $0<c\le 1$, we denote the square $c$-homothetic to $Q_j$ with respect to its center. Then we have
\[
 \int\int_{S_j} d_{S_j}(x_1,x_2)\, dx_1dx_2\ge \sup_{c\in (0,1]}\int\int_{cQ_j} d_{S_j}(x_1,x_2)\, dx_1dx_2=\sup_{c\in (0,1]} (1-c)c^2 r_j^3 =
 \frac{4}{27} r_j^3.
\]
This an the previous inequality imply
\begin{equation}\label{eq3.6}
\frac{4}{27} r_{\Gamma_F}^3\cdot\sum_{j=1}^m |\tilde n_j|\le \frac{2\pi d_{\Gamma_F}}{k} A_{\Gamma_F}\cdot \sum_{j=1}^m |\tilde n_j|+\left|\sum_{j=1}^m \tilde n_j \int\int_{S_j} p_k(x_1,x_2)\, dx_1dx_2\right|.
\end{equation}
Assuming that the second term on the right equals zero but $\sum_{j=1}^m |\tilde n_j|\ne 0$ we obtain, see \eqref{const},
\[
k\le \frac{27\pi}{2}\frac{A_{\Gamma_F} d_{\Gamma_F}}{r_{\Gamma_F}^3}< N_{\Gamma_F}.
\]
Therefore from \eqref{eq2.1}--\eqref{eq3.6} and the previous inequality we deduce that
vanishing of  moments 
\[
\int_a^b g_{1}(t)^{d_1}\cdot g_{2}(t)^{d_2}\cdot g'_{2}(t)\,dt\quad  \text{with}\quad \max\{d_1-1,d_2\}\le N_{\Gamma_F}
\]
implies that all $\tilde n_j= 0$ and so all moments from $g_1,g_2$ vanish.

The same argument (with the same conclusion) works for all $n_j$ in \eqref{eq2.2} being real numbers. In particular, since the moments on the right-hand side of \eqref{eq2.2} are linear functionals on
$H_1(\Gamma_F,\Re)$, the family 
\begin{equation}\label{eq3.7}
\left\{\sigma\mapsto \int_{\sigma}x_1^{d_1}x_2^{d_2} dx_2;\ \sigma\in H_1(\Gamma_F,\Re)\right\}\quad  \text{with}\quad \max\{d_1-1,d_2\}\le N_{\Gamma_F}
\end{equation}
forms the set of generators in the real vector space  $M_{\Gamma_F}$ generated by all such moments; moreover, the dimension of this space is $m:={\rm dim}\,H_1(\Gamma_F,\Re)$. (For otherwise, there exists a nonzero element $\sigma=\sum_{j=1}n_j[\ell_j]\in H_1(\Gamma_F,\Re)$ such that all moments in \eqref{eq3.7} vanish at $\sigma$, a contradiction.)

Let $\varphi_1,\dots,\varphi_m$ be a basis in $M_{\Gamma_F}$ formed by some moments of the family \eqref{eq3.7}. Then all other moments in $M_{\Gamma_F}$ can be expressed as linear combinations of $\varphi_1,\dots,\varphi_m$ with coefficients depending on $\Gamma_F$ only. This proves the main statement of the theorem:

For every  Lipschitz map $G=(g_1,g_2): [a,b]\rightarrow\Re^2$, $G(a)=G(b)$, with image in $\Gamma_F$
each moment from $g_1,g_2$ can be expressed as a finite linear combination with real coefficients  depending on $\Gamma_F$ only of $m$ moments  from the family
\[
\int_a^b g_{1}(t)^{d_1}\cdot g_{2}(t)^{d_2}\cdot g'_{2}(t)\,dt\quad  \text{with}\quad \max\{d_1-1,d_2\}\le \widetilde N
\]
for some $\widetilde N\le N_{\Gamma_F}$.

To complete the proof it remains to show that the optimal number $\widetilde N$  here cannot be less than $\lfloor\sqrt{m}-\frac 12\rfloor$. But assuming, on the contrary, that $\widetilde N< \lfloor\sqrt{m}-\frac 12\rfloor$ we easily obtain from the statement of the theorem, using duality between $H_1(\Gamma_F,\Re):=H_1(\Gamma_F)\otimes\Re$ and $M_{\Gamma_F}$, that the latter space is generated by the family 
\[
\left\{\sigma\mapsto \int_{\sigma}x_1^{d_1}x_2^{d_2} dx_2;\ \sigma\in H_1(\Gamma_F,\Re)\right\}\quad  \text{with}\quad \max\{d_1-1,d_2\}< \left\lfloor\sqrt{m}-\frac 12\right\rfloor.
\]
The cardinality of this family does not exceed $(\sqrt{m}+\frac 12)\cdot (\sqrt{m}-\frac 12)-1< m={\rm dim}\, M_{\Gamma_F}$, a contradiction.

The proof of the theorem is complete.
\section{Proof of Theorem \ref{te3.1}}
We exclude the case of $\Gamma_F$ homotopically trivial, because in this case all moments from coordinates of maps $G\in\mathcal H$ vanish,
see Proposition \ref{prop3.1}.
 
Next, we have
\begin{equation}\label{eq6.7'}
\int_a^b g_1(t)^{d_1}\cdots g_n(t)^{d_n}\cdot g_i'(t)\, dt=\int_G x_1^{d_1}\cdots x_n^{d_n} dx_i,\quad i=1,\dots, n.
\end{equation}
Here $x_1,\dots, x_n$ are coordinates on $\Re^n$ and $\int_G$ is the integral over the path $G$ (it is well defined because $G$ is Lipschitz).
Let $\ell_1,\dots , \ell_m : [0, 1]\rightarrow\Gamma_F$ be simple closed
Lipschitz curves in $\Re^n$ such that their images $[\ell_1],\dots, [\ell_m]$ in $H_1(\Gamma_F)$ are generators of this group.  
Let $[G]$ be the image of the closed path $G$ in $H_1(\Gamma_F)$.
Then $[G]=\sum_{j=1}^m n_j [\ell _j]$ for some $n_j\in\Z$, and since the integrands on the right-hand sides of \eqref{eq6.7'} are $d$-closed 1-forms on $\Gamma_F$ and curves $\ell_j$ are Lipschitz, 
\begin{equation}\label{eq6.8'}
\int_G x_1^{d_1}\cdots x_n^{d_n} dx_i=\sum_{j=1}^m n_{j}\int_{\ell_j} x_1^{d_1}\cdots x_n^{d_n} dx_i,\quad 1\le i\le n.
\end{equation}

Let $\mathbb F\subset\Re$ be the minimal field containing $\Q$ and all numbers $\int_{\ell_j} x_1^{d_1}\cdots x_n^{d_n} dx_i$ for all possible $i,j$ and $d_1,\dots, d_n\in\Z_+$. By definition, ${\rm card}\,\mathbb F=\aleph_0$ and all moments from coordinates of maps $G\in\mathcal H$ belong to $\mathbb F$.

Let $V\subset\Re^n\times\Re^n$ consist of vectors $v=(v_1,v_2)$ whose coordinates are algebraically independent over $\mathbb F$, i.e., for every nonzero polynomial $p$ on $\Re^n\times\Re^n\, (=\Re^{2n})$ with coefficients in $\mathbb F$, $p(v)\ne 0$. Since the cardinality of the set of such polynomials is $\aleph_0$ and their zero sets are of codimension at most one in $\Re^n\times\Re^n$, the set $(\Re^n\times\Re^n)\setminus V$ is of Lebesgue measure zero.

Let $v=(v_1,v_2)\in V$ and $v_i=(v_{i1},\dots, v_{in})$, $v_{ij}\in\Re$, $1\le j\le n$, $i=1,2$. Then for $G_v(t)=(g_{1v}, g_{2v}):=(\langle v_1,G(t)\rangle,\langle v_2,G(t)\rangle)$, $t\in [a,b]$, and $d\in\Z_+$ we have
\begin{equation}\label{eq6.7}
\begin{array}{l}
\displaystyle
\int_a^b g_{1v}(t)^d\cdot g_{2v}'(t)\, dt=\int_a^b\left(\sum_{j=1}^n v_{1j}g_j(t)\right)^d\cdot\left(\sum_{j=1}^n v_{2j}g_j'(t)\right) dt\\
\\
\displaystyle
=\int_a^b\left(\,\sum_{d_1+\cdots + d_n=d}\frac{d!}{d_1!\cdots d_n!}(v_{11} g_1)^{d_1}\cdots (v_{1n}g_{n})^{d_n}\right)\cdot\left(\sum_{j=1}^n v_{2j}g_j'(t)\right) dt.
\end{array}
\end{equation}
If all moments on the left-hand side of \eqref{eq6.7} equal to zero for $d\le N_v$, then, by the definition of the set $V$ and the previous identity, all moments
of degrees $d\le N_v$  from $g_1,\dots, g_n$ equal to zero, and vice versa. This shows that the constants $N$ and $N_v$ in the corresponding moments finiteness problems coincide and completes the proof of the theorem.

\sect{Proof of Theorem 3.4}
Similarly to \eqref{eq6.7'} and \eqref{eq6.8'} for every $G=(g_1,\dots, g_n)\in\mathcal G_F$ (see the Introduction) we have
\begin{equation}\label{eq7.1}
\int_a^b g_1(t)^{d_1}\cdots g_n(t)^{d_n}\cdot g_i'(t)\, dt=\sum_{j=1}^k n_j\cdot\int_{h_j} x_1^{d_1}\cdots x_n^{d_n} dx_i,\quad 1\le i\le n,
\end{equation}
for some $n_j\in\Z$.

Next, for a cube $K_i$ satisfying condition (*) of the theorem $h_{i}^{-1}(K_i\cap\gamma_i)=(a_i,b_i)\subset [0,1]$ (because $h_i$ is injective). If $h_i=(h_{i1},\dots, h_{in})$, then $h_{is_i}: (a_i,b_i)\rightarrow\Re$ is an injective Lipschitz map. In particular, $h_{is_i}$ is monotonic, i.e., $h_{is_i}'$ is either nonnegative or nonpositive a.e. on $(a_i,b_i)$. We define $\delta_{s_i}:=1$ in the first case and $\delta_{s_i}:=-1$ in the second one. 

Making the rearrangement of coordinate indices, if necessary, without loss of generality we may assume that there exist integers $m\in [1,n]$ and $0:= k_0<k_1<\cdots< k_m:=k$ such that
$s_i=r+1$ for all $k_{r}+1\le i\le k_{r+1}$, $r=0,\dots, m-1$.

Let $d_{K_i}$ be the distance to the closed set $\Re^n\setminus K_i$ in the $\ell_\infty^n$ metric on $\Re^n$.
Then $d_{K_i}$ is a Lipschitz function with respect to this metric with Lipschitz constant one.
We set
\[
d_i:={\rm sgn}\, n_i\cdot\delta_{s_i}\cdot d_{K_i},\quad 1\le i\le k,\quad\text{and}\quad \tilde d_r:=\sum_{j=1}^{k_r-k_{r-1}} d_{k_{r-1}+j},\quad 1\le r\le m.
\]
By the definition of the function $d_{K_i}$ and the Rademacher theorem, the derivative of $d_i$ restricted to a line $L$ parallel to one of the coordinate axes exists a.e. and its modulus is bounded by one on $L\cap K_i$ and equals zero on $L\setminus K_i$. From here and the fact that
the cubes $K_i$ are pairwise disjoint we obtain that the derivative of each $\tilde d_r|_L$ exists a.e. and its modulus is bounded by one, i.e., $\tilde d_r|_L$ is Lipschitz with Lipschitz constant bounded by one.
Applying to $\tilde d_r$ Proposition \ref{approx} we find a sequence of real polynomials $\{p_{r,k}\}_{ k\in\N,\, 1\le r\le m}$, where $p_{r,k}$ is of degree $k$ in each variable, such that
\[
\|\tilde d_r-p_{r,k}\|_{C(K)}\le \frac{n\pi d_{\Gamma_F}}{k}.
\]
(Here $K$ is the minimal closed cube containing images of all curves $h_i$.)

Hence, we have
\[
\begin{array}{l}
\displaystyle
\sum_{r=1}^m\left(\sum_{j=1}^{k_r-k_{r-1}} |n_j|\cdot\delta_r\cdot \int_{h_j} d_{K_j}(x_1,\dots, x_n)\, dx_{r}\right)=
\left|\sum_{j=1}^{k}n_j\cdot \int_{h_j}\sum_{r=1}^m \tilde d_r(x_1,\dots, x_n)\, dx_r\right|\\
\\
\displaystyle
\le \sum_{j=1}^{k} |n_j|\cdot \int_{h_j} \sum_{r=1}^m|\tilde d_r(x_1,\dots, x_n)-p_{r,k}(x_1,\dots,x_n)|\cdot |dx_r|\\
\\
\displaystyle
+\left|\sum_{j=1}^{k}n_j \int_{h_j} \sum_{r=1}^m p_{r,k}(x_1,\dots,x_n)\, dx_r\right|
\le \frac{n\pi d_{\Gamma_F}}{k}\left(\sum_{j=1}^k | n_j|\cdot \int_{h_j}\sum_{r=1}^m |dx_r|\right)\\
\\
\displaystyle
+
\left|\sum_{j=1}^k n_j \int_{h_j} \sum_{r=1}^m p_{r,k}(x_1,\dots, x_n)\, dx_r\right|\\
\\
\displaystyle
\le \frac{n\pi L_{\mathcal T} d_{\Gamma_F}}{k}\cdot \sum_{j=1}^k | n_j|+
\left|\sum_{j=1}^k n_j \int_{h_j} \sum_{r=1}^m p_{r,k}(x_1,\dots, x_n)\, dx_r\right|.
\end{array}
\]
On the other hand, in notation of the theorem we have
\[
\begin{array}{l}
\displaystyle
\sum_{r=1}^m\left(\sum_{j=1}^{k_r-k_{r-1}} |n_j|\cdot\delta_r\cdot \int_{h_j} d_{K_j}(x_1,\dots, x_n)\, dx_{r}\right)\ge
\sum_{r=1}^m\left(\sum_{j=1}^{k_r-k_{r-1}} |n_j|\cdot\frac{r_j\cdot l_j}{2} \right)\\
\\
\displaystyle
\ge \frac{r_{\mathcal T}l_{\mathcal T}}{2}\cdot\sum_{j=1}^k |n_j|.
\end{array}
\]
Combining the last two inequalities we obtain
\begin{equation}\label{eq7.2}
\frac{r_{\mathcal T}l_{\mathcal T}}{2}\cdot\sum_{j=1}^k |n_j|\le \frac{n\pi L_{\mathcal T} d_{\Gamma_F}}{k}\cdot\sum_{j=1}^k | n_j|+\left|\sum_{j=1}^k n_j \int\int_{h_j} \sum_{r=1}^m p_{r,k}(x_1,\dots, x_n)\, dx_r\right|.
\end{equation}
Assuming that the second term on the right-hand side equals zero but $\sum_{j=1}^m |n_j|\ne 0$ we obtain
\[
k\le\frac{2\pi n L_{\mathcal T} d_{\Gamma_F}}{r_{\mathcal T}l_{\mathcal T}}<\bar{N}_{\mathcal T}.
\]
Therefore from \eqref{eq7.1} and the previous inequality we deduce that for each $G=(g_1,\dots, g_n)\in\mathcal G_F$ vanishing of moments
\[
\int_a^b g_1(t)^{d_1}\cdots g_n(t)^{d_n}\cdot g_i'(t)\, dt\quad  \text{with}\quad \max_{1\le j\le n}\{d_j\}\le \bar{N}_{\mathcal T},\quad 1\le i\le n,
\]
implies that all $n_j= 0$ and so all moments from $g_1,\dots, g_n$ vanish.

The same argument (with the same conclusion) works if all $n_j$ on the right-hand side of \eqref{eq7.1} are real numbers. 

Now, let
$\ell_1,\dots,\ell_m: [0,1]\rightarrow C\subset\Gamma_F$ be simple closed Lipschitz curves whose images $[\ell_1],\dots,[\ell_m]\in H_1(\Gamma_F)$ form the set of generators of this group. Then each element $\sigma \in H_1(\Gamma_F,\Re):=H_1(\Gamma_F)\otimes\Re$ is presented in a unique way as
$\sigma=\sum_{j=1}^mc_j[\ell_j]$ for some $c_j\in\Re$. We have, for each $1\le i\le n$,
\[
\sum_{j=1}^m c_j\cdot\int_{\ell_j} x_1^{d_1}\cdots x_n^{d_n} dx_i=\sum_{j=1}^k \tilde c_j\cdot\int_{h_j} x_1^{d_1}\cdots x_n^{d_n} dx_i,
\]
where each $\tilde c_j$ is a linear combination with integer coefficients of some $c_i$. Moreover, all $c_j$ are equal to zero as soon as all $\tilde c_j$ are equal to zero. Thus vanishing of all moments
\[
\sum_{j=1}^m c_j\cdot\int_{\ell_j} x_1^{d_1}\cdots x_n^{d_n} dx_i\quad  \text{with}\quad \max_{1\le j\le n}\{d_j\}\le \bar{N}_{\mathcal T},\quad 1\le i\le n,
\]
implies that $\sigma=0$.

We finish the proof repeating literally the arguments from the proof of Theorem \ref{te2.1}. Specifically, we obtain that the family
\begin{equation}\label{eq7.3}
\left\{\sigma\mapsto \int_{\sigma}x_1^{d_1}\cdots x_n^{d_n} dx_i;\ \sigma\in H_1(\Gamma_F,\Re)\right\}\ \text{with}\ \max_{1\le j\le n}\{d_j\}\le \bar{N}_{\mathcal T},\ 1\le i\le n,
\end{equation}
forms the set of generators in the real vector space  $M_{\Gamma_F}$ generated by all such moments and its dimension is $m:={\rm dim}\,H_1(\Gamma_F,\Re)$. Let $\varphi_1,\dots,\varphi_m$ be a basis in $M_{\Gamma_F}$ formed by some moments of the family \eqref{eq7.3}. Then all other moments in $M_{\Gamma_F}$ can be expressed as linear combinations of $\varphi_1,\dots,\varphi_m$ with coefficients depending on $\Gamma_F$ only. This proves the main statement of the theorem. 

Further, assuming, on the contrary, that $\widetilde N< \max\left\{\left\lfloor\sqrt[n]{\frac mn}-1\right\rfloor, 1\right\}=:A$ we easily obtain from the statement of the theorem, using duality between $H_1(\Gamma_F,\Re)$ and $M_{\Gamma_F}$, that the latter space is generated by the family 
\[
\left\{\sigma\mapsto \int_{\sigma}x_1^{d_1}\cdots x_n^{d_n} dx_i;\ \sigma\in H_1(\Gamma_F,\Re)\right\}\quad  \text{with}\quad \max_{1\le j\le n}\{d_j\}<A,\quad 1\le i\le n.
\]
The cardinality of this family does not exceed $m={\rm dim}\, M_{\Gamma_F}$, a contradiction.

The proof of the theorem is complete.

\section{Proofs of Results of Section 4}

\begin{proof}[Proof of Theorem \ref{teo4.2} ]
As was mentioned after the formulation of the theorem, it suffices to prove the implication (d)$\Rightarrow$(b) only. So let $F=(f_1,\dots, f_n): [a,b]\rightarrow\Re^n$ be a Lipschitz map such that $\Gamma_F$ is Lipschitz triangulable and traversable and $F$ is a closed path that covers an Eulerian trail in $\Gamma_F$ and represents zero element of the homology group $H_1(\Gamma_F)$. This means that $F=\tilde F\circ\tilde f$ for continuous maps $\tilde f: [a,b]\rightarrow\mathbb S$ and $\tilde F: I:=\tilde f([a,b])\rightarrow \Gamma_F$ such that $\tilde F$ is surjective and injective outside a finite number of points of $I$. 

Let $p:\Re:=\Re/\Z:=\mathbb S$ be the universal covering of $\mathbb S$. By the covering homotopy theorem, see, e.g., \cite{Hu}, there exists a path $f':[a,b]\rightarrow\Re$, $f'(a):=0$, such that $\tilde f=p\circ f'$. Consider the function
$g: [a,b]\rightarrow f'([a,b])$, $g(s):=\frac{f'(b)}{b-a}(s-a)$, $s\in [a,b]$. Then 
$H(t)=t\cdot f'+(1-t)\cdot g$, $0\le t\le 1$, is a homotopy between $f'$ and $g$ fixing points $f'(a)$ and $f'(b)$ and such that the range of each $H(t)$ belongs to $f'([a,b])$. In turn, since $F(a)=F(b)$, the composition $\tilde F\circ p\circ H$ is a homotopy between closed paths $F: [a,b]\rightarrow\Gamma_F$ and
$\tilde g:=\tilde F\circ p\circ g: [a,b]\rightarrow\Gamma_F$. In particular, $\tilde g$ represents zero element of $H_1(\Gamma_F)$ as well.

Set $n:=\lfloor f'(b) \rfloor$ and $l:=f'(b)-n$. Then $p\circ g: [a,b]\rightarrow\mathbb S$ is the composition of two paths, one represents the element $n\cdot e$ in the fundamental group $\pi_1(\mathbb S, \tilde f(a))$, where $e$ is the generator of this group, and another one, say, $\gamma$, is a simple proper arc of length $l$ in $\mathbb S$ with endpoints $\tilde f(a)$, $\tilde f(b)$. 

First, suppose that $n\ne 0$ and $l\ne 0$. Then $\tilde f$ is surjective and, in particular, the domain of $\tilde F$ is $\mathbb S$. Since $\tilde F$ is an Eulerian trail in $\Gamma_F$, we obtain immediately that $\tilde F\circ e$ and $\tilde  F\circ\gamma$ are closed paths in $\Gamma_F$ representing linearly independent elements $h_1, h_2\in H_1(\Gamma_F,\Q)$. Hence, $\tilde g$ represents the element $n h_1+h_2=0\in H_1(\Gamma_F,\Q)$. This yields $h_2=0$ and $n=0$, a contradiction. Similarly, the cases $n\ne 0$, $l=0$ and $n=0$, $\ell\ne 0$ lead to contradictions as well. Thus we have $n=0$ and $l=0$. This shows that $\tilde f$ is a closed contractible path in $\mathbb S$ and $f'$ is a closed path in $\mathbb R$ that covers $\tilde f$, i.e., $F=(\tilde F\circ p)\circ f'$, where $f': [a,b]\rightarrow\Re$ satisfies $f'(a)=f'(b)$ and this is statement (b).

The proof of the theorem is complete.
\end{proof}

\begin{proof}[Proof of Theorem \ref{teo4.4.1}]
According to the formula for the first return map of equation \eqref{eq4.3}, see, e.g., \cite{B1}, the necessary condition for equation
\[
\frac{dv}{dt}=f_1' v^{2}+f_2' v^3,
\]
where $F:=(f_1,f_2):[a,b]\rightarrow\Gamma_F\Subset\Re^2$, $F(a)=F(b)=0$, to determine a center is
\[
3\cdot\int_a^b f_1(t)\cdot f_2'(t)\, dt+2\cdot\int_a^b f_2(t)\cdot f_1'(t)\, dt=0.
\] 
(The expression on the left-hand side is the 4th coefficient in the power series decomposition of the first return map).
We have
\[
\int_a^b f_1(t)\cdot f_2'(t)\, dt+\int_a^b f_2(t)\cdot f_1'(t)\, dt=\bigl(f_1(t)\cdot f_2(t)\bigr)|_a^b=0.
\]
Thus the above condition is equivalent to
\begin{equation}\label{eq9.4}
\int_a^b f_1(t)\cdot f_2'(t)\, dt=0.
\end{equation}
From here, the hypothesis of the theorem and Proposition \ref{prop2.2} we obtain that all moments from $f_1$ and $f_2$ vanish. But then, according to Theorem \ref{teo4.2}, the above Abel differential equation determines a universal center. This shows that condition \eqref{eq9.4} is also sufficient for the Abel equation satisfying assumptions of the theorem to determine a center, and all centers of such equations are universal.
\end{proof}

\begin{proof}[Proofs of Theorem \ref{teo4.5} and \ref{teo4.6}]
The proofs follow straightforwardly from Theorems \ref{te2.1}, \ref{te3.4} and \ref{teo4.2}.
\end{proof}

\end{document}